\documentclass[10pt, final, journal, letterpaper, oneside, twocolumn]{IEEEtran}

\usepackage{tikz}
\usetikzlibrary{positioning} 
\usepackage[tbtags]{amsmath}
\usepackage{enumerate}
\usepackage{graphicx}
\usepackage{amsfonts}
\usepackage{caption}
\usepackage{subcaption}
\usepackage{xcolor}

\DeclareMathOperator{\blkdiag}{blkdiag}
\DeclareMathOperator{\col}{col}

\newtheorem{thm}{Theorem}
\newtheorem{lemma}{Lemma}

\newtheorem{defn}{Definition}

\newtheorem{remark}{Remark}
\newtheorem{asmp}{Assumption}

\newcommand{\reals}{\mathbb{R}}

\tikzstyle{block} = [draw, fill=white, rectangle, 
    minimum height=2em, minimum width=2em]
\tikzstyle{sum} = [draw, fill=white, circle, node distance=0.5cm, inner sep=0pt, minimum size=0.25cm]
\tikzstyle{input} = [coordinate]
\tikzstyle{output} = [coordinate]
\tikzstyle{pinstyle} = [pin edge={to-,thin,black}]
\tikzstyle{branch}=[fill,shape=circle,minimum size=3pt,inner sep=0pt]

\begin{document}

\title{\LARGE{Dynamic NE Seeking for Multi-Integrator Networked Agents with Disturbance Rejection}}

\author{Andrew R.~Romano and Lacra Pavel%
\thanks{This work was supported by NSERC Grant (261764). 
	A.R. Romano and L. Pavel are with Department of Electrical and Computer Engineering, University of Toronto, Canada. 
	{\tt\small andrew.romano@mail.utoronto.ca, pavel@ece.utoronto.ca}}%
}

\maketitle

\begin{abstract}
In this paper, we  consider  game problems played by (multi)-integrator agents, subject to external disturbances. We propose Nash equilibrium seeking dynamics based on gradient-play, augmented with a dynamic internal-model based component, which is a reduced-order observer of the disturbance. We consider single-, double- and extensions to multi-integrator agents, 
in a partial-information setting, where agents have only partial knowledge on the others' decisions over a network. The lack of global information is offset by each agent maintaining an estimate of the others' states, based on local communication with its neighbours.  
Each agent has an additional dynamic component that drives its estimates to the consensus subspace. In all cases, we show convergence to the Nash equilibrium irrespective of disturbances. Our proofs leverage input-to-state stability under strong monotonicity of the pseudo-gradient and Lipschitz continuity of the extended pseudo-gradient.
\end{abstract}

\section{Introduction}
Game theory has 
found many applications in multi-agent engineering problems, wherein each agent can be modelled as an independent, selfish decision maker that tries to optimize its  individual, but coupled, cost function. These include
 wireless communication networks \cite{faawb06}, \cite{lh10}, \cite{ch12}, 
  optical networks \cite{lp12}, \cite{lp06}, 
 smart-grid and PEV charging \cite{mwjsl10}, \cite{sg16c}, \cite{hi13} noncooperative flow control \cite{ysm11}, \cite{ab05} and multi-agent formation problems\cite{lqs14}. 
 The relevant equilibrium sought is the Nash equilibrium (NE), whereby no agent has incentive to unilaterally change its action.  
 The objective is to design either continuous-time or discrete-time, distributed learning schemes that converge to the NE under reasonable assumptions on the game properties and agent knowledge. Most works focus on algorithms for agents that either do not have dynamics, 
 or have single integrator dynamics, and disturbances are not explicitly considered, \cite{FKB12}. 

There are many scenarios when the game or the agents are subject to disturbances, noise or uncertainties. Examples are demand-side management in smart-grids,  with changes in the energy consumption demand, \cite{mwjsl10}, feedback control for PEV charging load allocation,  \cite{hi13}, or power control for optical-signal-to-noise ratio (OSNR) in the presence of pilot tones, \cite{tp06}. Yet there have been relatively few works on Nash equilibrium seeking in such settings. 
In \cite{hi13}, a time-varying pricing function that affects the cost functions of each agent is considered. Only robustness to the time-varying component is investigated. 
Another good motivating example is the case of a group of mobile robots in a sensor network, similar to the examples in \cite{sjs12}. Each agent in the network has a goal related to its global position. 
However, it must also consider its position relative to the other agents in the network in order to make sure that it maintains communication with its neighbours. This can easily be formulated as a game played by the robots, which can be modelled as higher-order agents. In addition, each robot may be subject to a disturbance, e.g., wind or a slope in the terrain. It is important that these robots be able to reject this deterministic disturbance and still converge to the NE. A similar problem without disturbances was presented in \cite{zm13}, however the state space is discretized and the game is treated as a finite action game. This formulation ignores the dynamics of the individual agents.


Motivated by the above, in this paper our focus is to extend these results to games wherein the agents are modelled as (multi)-integrator systems subject to external deterministic disturbances. This is related to NE seeking with noisy feedback, on which there has been recent work. A dual-averaging algorithm with noisy gradients is considered in \cite{ms17}. A discrete-time extremum seeking algorithm with noisy cost measurements for agents modelled as single and double integrators and kinematic unicycles is investigated in \cite{sjs12}. In both of these papers, the noise involved is stochastic in nature instead of a deterministic disturbance, as we consider here. Separately, NE seeking in the special class of aggregative games for Euler-Lagrange systems has been recently investigated in \cite{dl19},  which is similar to our work due to the dynamic nature of the agents involved, but without disturbances being considered.

Our work is related to the literature on disturbance rejection and tracking in multi-agent systems, \cite{bdp15}, \cite{dpj14}, \cite{zd13}, \cite{xwhj16}, 
\cite{xlh17}.  
Most output regulation problems in multi-agent systems 
can be viewed as specific cases of game theoretical problems. The synchronization problem, for example, can be regarded as a special game where each agent's cost function is quadratic and corresponds to the sum of the squared distances to all of its neighbours. 

Our work is also related to distributed optimization, where a group of agents  cooperatively minimize a global cost function, the sum of the agents' individual cost functions. Optimization schemes that reject disturbances have been discussed for single integrator systems \cite{wyh14}, 
systems with unit relative degree \cite{whj16} and systems with double integrator dynamics \cite{twy18}. In \cite{CortesSICON2016}  the robustness of a continuous-time distributed optimization algorithm  is analyzed in the presence of  additively persistent noise on agents' communication and computation, in a directed communication graph. 
Key differences from a game setup are the cooperative nature of the problem and the fact that usually each agent's cost is decoupled of the others' variables. 
Exploiting summability, leads to a set of parallel decoupled optimization problems, one for each agent and its own cost function. 
Even when the overall cost is not separable, due to its summable structure, one can extend the problem to an augmented space of estimates, where it becomes separable and convex. In a game context, an agent's cost is inherently \emph{coupled} to the others' decisions, on which it does not have control and convexity is only partial.  

\textit{Contributions. } 
 Motivated by the above, in this paper we consider how to design Nash equilibrium seeking dynamics  
  that simultaneously reject exogenous  disturbances. We consider single and double-integrator agents, i.e., agents that behave as continuous-time dynamical systems that integrate their respective inputs, 
 in a partial-information setting, i.e., networked regimes where agents may only access the states of their neighbours. We also discuss extensions to multi-integrator agents. 
Unlike  multi-agent set stabilization problems with disturbance rejection, herein the stabilization goal is the a priori unknown Nash equilibrium of the game, which  has to be reached irrespective of  disturbances. 
In all cases, we make standard assumptions that provide existence and uniqueness of the NE of the game.


Due to the partial-information setting, we inspire from the disturbance-free  results in 
\cite{gp18}. 
Each player keeps track of an  estimate of the others' decisions as in  \cite{gp18}, and the problem can be seen as one of  multi-agent agreement with disturbance rejection. The agreement subspace is the estimate-consensus subspace at the Nash equilibrium, irrespective of the disturbance. The proposed agent learning dynamics has two components: a gradient-play with estimate consensus component (that drives each player's dynamics towards minimizing its own cost function) and a dynamic internal-model component, which effectively implements a reduced-order observer of the disturbance. Unlike  typical multi-agent agreement, \cite{bdp15}, \cite{dpj14},  \cite{xwhj16}, 
we cannot use individual passivity of each agent.  
Rather, our proofs rely on combining input-to-state stability with  design of a reduced-observer for  disturbance, under strong monotonicity of the pseudo-gradient and  Lipschitz continuity of the extended pseudo-gradient.    
The resulting  agent dynamics are locally distributed, with coupling introduced only through the communication graph. 



The paper is organized as follows. 
In Section \ref{sec:background}, we give the necessary background on nonlinear systems, graph theory and noncooperative game theory. In Section \ref{sec:problem}, we formulate the NE seeking problem for dynamic agents with disturbance rejection. In Section \ref{sec:SingleIntegrator}, we give our results for NE seeking dynamics with disturbance rejection for single-integrator agents. 
In Section \ref{sec:second_integrator}, we formulate a NE seeking algorithm for double-integrator agents and discuss extensions to multi-integrator agents. In Section \ref{sec:simulations}, we compare by simulation their performance with those of a standard gradient-play dynamics and an augmented gradient-play dynamics with estimate consensus (partial information setting), and give conclusions in Section \ref{sec:conclusions}. 
A short version of this work appeared in \cite{AR_LP_CDC2018}, 
where only single-integrators are treated. 

\emph{Notations.} Let $\reals$, $\reals_{\geq 0}$ denote the set of real and non-negative real numbers, $\mathbb{C}$ and $\mathbb{C}^-$ the set of complex numbers and complex numbers with negative real part. 
Given $x, y \!\in  \!\reals^n$, $x^Ty$ denotes the inner product of $x$ and $y$. Let $\|\cdot\|\!:\!\reals^n \!\rightarrow \!\reals_{\geq 0}$ denote the Euclidean norm and $\|\cdot\|\!:\!\reals^{m\times n} \!\rightarrow \!\reals_{\geq 0}$ denote the induced matrix norm. 
$\col(x_1,\dots,x_N)$  denotes $[x_1^T,\dots,x_N^T]^T$. Given matrices $A_1,\dots,A_N$, $\blkdiag(A_1,\dots,A_N)$ denotes the block diagonal matrix with $A_i$ on the diagonal. $I_n$ denotes the $n\!\times \! n$ identity matrix. $\boldsymbol{1}_n$ denotes the $n\!\times \!1$ all ones vector.  
$A\otimes B$ denotes the Kronecker product of matrices $A$ and $B$. 

\vspace{-0.25cm}
\section{Background}\label{sec:background}

\subsection{Input to State Stability}

In this work, we model the dynamics of each agent as a continuous time dynamical system. We first introduce some background from \cite{hk02}. Consider a nonlinear system, 
\vspace{-0.23cm}\begin{align} \label{eq:NLSys}
\dot x = f(x,u)
\end{align}
where $\dot x:=\frac{dx(t)}{dt}$, $f\!:\!\reals^n \!\times \!\reals^m \!\rightarrow \!\reals^n$ is locally Lipschitz in $x$ and $u$ and the input $u(t)$ is a piecewise continuous, bounded function. 
\begin{defn} \label{defn:ISS}
System (\ref{eq:NLSys}) is input-to-state stable (ISS) if there exist $\beta \in \mathcal{KL}$ and $\gamma \in \mathcal{K}$ such that for any initial state $x(t_0)$ and any bounded input $u(t)$, the solution $x(t)$ satisfies 
\vspace{-0.25cm}\begin{align*}
\|x(t)\| \leq \beta(\|x(t_0)\|,t-t_0)+\gamma\Big(\sup_{t_o\leq\tau\leq t} \|u(\tau)\|\Big), \, \,\forall t\geq t_0
\end{align*}
\end{defn}

\begin{thm} \label{thm:ISSLyapunov}
(Theorem 4.19, \cite{hk02}) Let $V(x)$ be a continuously differentiable function such that
\vspace{-0.23cm}\begin{align*}
\alpha_1(\|x\|)&\leq V(x)\leq\alpha_2(\|x\|)\\
\frac{\partial V}{\partial x}f(x,u) &\leq - W(x),\ \forall \|x\| \geq \rho(\|u\|) > 0
\end{align*}
$\forall x \!\in \!\reals^n,\ u \!\in \!\reals^m$, where $\alpha_1,\alpha_2  \!\in \!\mathcal{K}_\infty, \!\rho \! \in \! \mathcal{K}$, and $W(x)$ is 
positive definite. Then system (\ref{eq:NLSys}) is ISS with $\gamma  \!=\! \alpha_1^{-1}\! \circ \! \alpha_2 \! \circ  \!\rho$.
\end{thm}

Consider now the cascade of two systems
\vspace{-0.28cm}\begin{align} \label{eq:NLSysCascade1}
\dot x_1 &= f_1(x_1,x_2)\\
\label{eq:NLSysCascade2}
\dot x_2 &= f_2(x_2)
\end{align}
with $f_1 \!: \!\reals^{n_1}  \!\times  \!\reals^{n_2}  \!\rightarrow  \!\reals^{n_1}$, $f_2 \!: \!\reals^{n_2}  \!\rightarrow  \!\reals^{n_2}$ locally Lipschitz. 
\begin{lemma} \label{lemma:ISSCascade}
(Lemma 4.7, \cite{hk02}) If the system (\ref{eq:NLSysCascade1}) with $x_2$ as an input is ISS and the origin of (\ref{eq:NLSysCascade2}) is globally uniformly asymptotically stable, then the origin of the cascade system (\ref{eq:NLSysCascade1}) and (\ref{eq:NLSysCascade2}) is globally uniformly asymptotically stable.
\end{lemma}

\subsection{Graph Theory}
In this paper, we consider NE seeking for dynamic agents with communication over networks with fixed (static) topology. The communication protocol relies on graph theory.
The following is from \cite{gr01}. An undirected graph, $G$ is a pair $G = (\mathcal{I},E)$ where  $\mathcal{I} = \{1,\dots,N\}$  is the vertex set and $E \subset \mathcal{I}\times\mathcal{I}$ is the edge set. Since $G$ is an undirected graph, for all $i,j \in \mathcal{I}$, if $(i,j)\in E$ then $(j,i)\in E$. Let $\mathcal{N}_i \subset \mathcal{I}$  denote the set of neighbours of player $i$. The adjacency matrix $\textbf{A} = [a_{ij}] \in \reals^{N \times N}$ of the graph $G$ is defined such that $a_{ij}=1$ if $(j,i)\in E$ and $a_{ij}=0$ otherwise. For an an undirected graph, $a_{ij}=a_{ji}$. $G$ is connected if any two agents are connected by a path. The Laplacian matrix $L = [l_{ij}] \in \reals^{N \times N}$ of the graph $G$ is defined as $l_{ii} = \sum_{j\neq i} a_{ij}=|\mathcal{N}_i|$ and $l_{ij} = -a_{ij}$, for $i\neq j$. 
For an undirected and connected graph, $L$ is symmetric positive definite and has a simple zero eigenvalue such that $0<\lambda_2(L)\leq \ldots \leq \lambda_N(L)$ and $L\textbf{1}_N = 0$. Furthermore, for any vector $y \in \reals^N$ satisfying $\textbf{1}_N^Ty=0$, $\lambda_2(L)\|y\|^2\leq y^TLy \leq \lambda_N(L)\|y\|^2$.

\subsection{Game Theory}
Consider a set of players, $\mathcal{I} = \{1,\dots,N\}$. Each player $i\in \mathcal{I}$ controls its own action $x_i \in \Omega_i \subset \reals^{n_i}$. The overall action set of the players is $\Omega = \Omega_1 \times \dots \times \Omega_N \subset \reals^n$, where $n = \sum_{i\in \mathcal{I}} n_i$. Let $x = (x_i,x_{-i})\in \Omega$ denote the overall action profile of all players, where $\textcolor{blue}{x_{-i}} \in \Omega_{-i} = \Omega_1 \times \dots \times \Omega_{i-1} \times \Omega_{i+1} \times \dots \times \Omega_N\subset \reals^{n_{-i}}$ is the action set of all players except for player $i$. Let $J_i:\Omega \rightarrow \reals$ be the cost function of player $i$. Each player tries to minimize its own cost function over its action. Denote the game $\mathcal{G}(\mathcal{I},J_i,\Omega_i)$. \begin{defn} \label{def:NE}
Given a game $\mathcal{G}(\mathcal{I},J_i,\Omega_i)$, an action profile $x^* = (x_i^*,x_{-i}^*)\in \Omega$ is a Nash Equilibrium (NE) of $\mathcal{G}$ if
\begin{align*}
J_i(x_i^*,x_{-i}^*)\leq J_i(x_i,x_{-i}^*) \quad \forall i \in \mathcal{I},\ \forall x_i \in \Omega_i
\end{align*}
At a Nash Equilibrium no player can unilaterally decrease its cost, and thus has no incentive to switch strategies (actions) on its own.
\end{defn}
\begin{asmp} \label{asmp:Jsmooth}
		For each $ i \in \mathcal{I}$, let $\Omega_i=\reals^{n_i}$, the cost function $J_i:\Omega \rightarrow \reals$ be $\mathcal{C}^1$ in its arguments and convex in $x_i$.
\end{asmp}

Under Assumption \ref{asmp:Jsmooth}, 
any NE satisfies
\vspace{-0.26cm}\begin{align} \label{eq:NashInner}
 \nabla_{i} J_i(x^*_i,x^*_{-i})=0,\quad \forall i \in \mathcal{I}
\end{align}
where $\nabla_i J_i(x_i,x_{-i}) = \frac{\partial}{\partial x_i} J_i(x_i,x_{-i}) \in \reals^{n_i}$ is the partial gradient of player $i$'s cost function, with respect to its own action. We denote the set of all NE in the game by   \vspace{-0.25cm}
\begin{align}
\Gamma_{NE} = \big \{x \in  \reals^n| \nabla_i J_i(x_i,x_{-i}) = 0,\, \forall i \in \mathcal{I}&\big \}
\end{align} 
Let  $\!F(x)\!\! =\!\! \col(\!\nabla_1 J_1(x),\dots,\!\nabla_N J_N(x))$ denote the pseudo-gradient- the stacked vector of all partial gradients, so  (\ref{eq:NashInner}) is  \vspace{-0.5cm}
\begin{align*} 
 F(x^*) = 0
\end{align*}
\begin{asmp} \label{asmp:PseudoGrad}
		The pseudo-gradient $F:\Omega \rightarrow \reals^n$ is strongly monotone, $(x-x')^T(F(x)-F(x'))>\mu||x-x'||^2$, $\forall x,x' \in \reals^n$ for $\mu>0$ and Lipschitz continuous, $||F(x)-F(x')|| \leq \theta ||x-x'||$, $\theta >0$.
\end{asmp}
Under Assumptions \ref{asmp:Jsmooth} and \ref{asmp:PseudoGrad}, by Theorem 3 in \cite{sfpp14}, the game has a unique NE.
 \vspace{-0.25cm}
\subsection{Full-Information Gradient Dynamics} \label{sec:GradDyn}
%

In the rest of this paper, we assume that each agent updates its action in a continuous manner, therefore $x_i=x_i(t)$. For simplicity of notation, we drop the explicit dependence on time. In a game with perfect information, i.e., complete communication graph, 
a gradient-based NE seeking algorithm (gradient-play) can be used for action update, given by
\vspace{-0.26cm}\begin{align}\label{eq:GradDyn}
\Sigma_i: \quad \dot x_i = -\nabla_i J_i(x_i,x_{-i}),\quad \forall i \in \mathcal{I}
\end{align}
We call $\Sigma_i$  the \emph{agent learning dynamics}, and note that it requires full decision information of the others', $x_{-i}$.

 The game can be visualized as an interconnection between all agents' learning dynamics, $\Sigma_i$, $i \!\in \!\mathcal{I}$, represented as in Fig. \ref{fig:GameNoDyn}, where $\Sigma_{-i}$ denote the other agents' learning dynamics (except $i$), and  $s_{-i}$ is the information received by agent $i$ from the others $\Sigma_{-i}$ in continuous-time. Hence, in the full information setting, $s_{-i} =x_{-i}$.  
\begin{figure}[h!]
\vspace{-0.26cm}
  \centering
\begin{tikzpicture}[auto, node distance=1cm]
    \node [input, name=input] {};
    \node [block, right = of input, minimum width = 1 cm, minimum height = 0.75 cm] (player) {$ \Sigma_i $};

    \node [block, below right = -0.75 cm and 1 cm of player, minimum height = 2.5 cm, minimum width = 1 cm] (game) {$ \mathcal{G} $};
    
    \node [block, below = of player, minimum width = 1 cm, minimum height = 0.75 cm] (others) {$\Sigma_{-i}$};

\path (player.south) -- (player.south west) coordinate[pos=0.5] (a1);
\path (others.north) -- (others.north west) coordinate[pos=0.5] (b1);

\path (player.south) -- (player.south east) coordinate[pos=0.5] (a2);
\path (others.north) -- (others.north east) coordinate[pos=0.5] (b2);


    \draw [->] (a2) -- node [name=si] {$s_i$}(b2);
    \draw [->] (b1) -- node [name=so] {$s_{-i}$}(a1);

    \draw [->] (player) -- node [name=xi] {$x_i$}(player-|game.west);
    \draw [->] (others) -- node [name=xo] {$x_{-i}$}(others-|game.west);
\end{tikzpicture}
	\caption{Game as an interconnection between agents' dynamics} 
	\label{fig:GameNoDyn}
\end{figure}
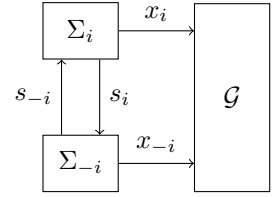
With \eqref{eq:GradDyn},  the  overall  dynamics of all players, $\Sigma = (\Sigma_i,\Sigma_{-i})$ is $\Sigma: \,\,\, \dot x = -F(x)$. 
Note that $\Sigma$ can be viewed as a feedback interconnection between a bank of integrators with the pseudo-gradient map $F$. 
Under Assumption \ref{asmp:Jsmooth}, the solutions of (\ref{eq:GradDyn}) exist and are unique for any initial condition, $x(0)$. Under Assumption \ref{asmp:PseudoGrad}, the unique Nash Equilibrium of the game is globally asymptotically stable for the  interconnected $\Sigma$,  with $\Sigma_i$ as in  \eqref{eq:GradDyn}, (cf. \cite{Flam02} or Lemma 1, \cite{gp18}).

\subsection{Partial-Information Gradient Dynamics} \label{sec:GradDynPartial}
Often only partial information is available to each agent, i.e., from the neighbours of each agent. In this case, a modified algorithm must be used, where agent $i$ uses estimates, $\textbf{x}^i$, which it shares with its neighbours, and evaluates its gradient using these estimates instead of the others' actions. Referring to Fig. \ref{fig:GameNoDyn}, in this case, $s_{-i} = \{\textbf{x}^j|j \in \mathcal{N}_i\}$. The following is from \cite{gp18}. Consider a game with information exchanged over a network, with static communication graph, $G_c$ and Laplacian, $L$.
\begin{asmp} \label{asmp:GraphConnected}
The undirected graph $G_c$ is connected.
\end{asmp}
Consider the following agent learning dynamics
\vspace{-0.26cm}\begin{align}\label{eq:GradDynPartialInfo}
\Sigma_i: \begin{cases}\dot{\textbf{x}}_{-i}^i = -\mathcal{S}_i \sum_{j \in \mathcal{N}_i} (\textbf x^i - \textbf x^j)\\
\dot x_i = -\nabla_i J_i(x_i,\textbf{x}_{-i}^i)  - \mathcal R_i \sum_{j \in \mathcal{N}_i} (\textbf x^i - \textbf x^j), \, \forall i \in \mathcal{I} \end{cases}\hspace{-0.26cm}
\end{align}
where $\textbf{x}_{-i}^i$ are agent $i$'s estimates of the others actions. 
Based on local communication with its neighbours, $\mathcal{N}_i$, each agent $i$ computes estimates of all other agents' actions, $\textbf{x}_{-i}^i\!\!=\!\!\col(\textbf{x}_{1}^i,\dots,\textbf{x}_{i-1}^i,\textbf{x}_{i+1}^i,\dots,\textbf{x}_{N}^i) \! \in \! \reals^{n_{-i}}$ and uses these estimates to evaluate its gradient, $\nabla_i J_i(x_i,\textbf{x}_{-i}^i)$. Then, $\textbf{x}^i\!=\!\col(\textbf{x}_{1}^i,\dots,\textbf{x}_{i-1}^i,x_i,\textbf{x}_{i+1}^i,\dots,\textbf{x}_{N}^i) \! \in \! \reals^n$ and $\textbf{x}\!=\!\col(\textbf{x}^1,\dots,\textbf{x}^N) \!\in\! \reals^{Nn}$. The matrices $\mathcal{R}_i$ and $\mathcal{S}_i$ are defined as follows 
\vspace{-0.26cm}\begin{align} \label{eq:RS}
\begin{split}
\mathcal R_i &:= \begin{bmatrix} 0_{n_i\times n_{<i}} & I_{n_i} & 0_{n_i\times n_{>i}} \end{bmatrix}\\
\mathcal{S}_i &:= \begin{bmatrix} I_{n_{<i}} & 0_{n_{<i} \times n_i} & 0_{n_{<i} \times n_{>i}}\\
                                       0_{n_{>i} \times n_{<i}} & 0_{n_{>i} \times n_i} &  I_{n_{>i}} \end{bmatrix}
\end{split}
\end{align}
for actions and estimates selection, where $n_{<i}:=\sum_{j<i\ j,i\in \mathcal{I}}n_j$ and $n_{>i}:=\sum_{j>i\ j,i\in \mathcal{I}}n_j$. Then $x_i = \mathcal{R}_i \textbf{x}^i$, $\textbf{x}_{-i}^i = \mathcal{S}_i \textbf{x}^i$, and $\textbf{x}^i = \mathcal{R}_i^T x_i+\mathcal{S}_i^T \textbf{x}^i_{-i}$. Note that, 
\vspace{-0.26cm}\begin{align}\label{eq:r_i}
\mathcal{R}_i^T\mathcal{R}_i+\mathcal{S}_i^T\mathcal{S}_i = I_{n}, \qquad \qquad \forall i \in \mathcal{I}
\end{align}


The vector of stacked, partial gradients $\nabla_i J_i(x_i,\textbf x_{-i}^i)$ in  \eqref{eq:GradDynPartialInfo}, computed based on estimates, is denoted as \vspace{-0.26cm}\begin{align} \label{eq:ExtPseudo}
\textbf F(\textbf{x}) = \col(\nabla_1 J_1(x_1,\textbf{x}_{-1}^1),\dots,\nabla_N J_N(x_N,\textbf{x}_{-N}^N)).
\end{align}
and is called the extended pseudo-gradient.  Note that $\textbf F$ satisfies  $\textbf{F}(\textbf{1}_N \otimes x)  = F(x) $ for any $x$, hence
\vspace{-0.26cm}\begin{align} \label{eq:ExtPseudoZero}
\textbf{F}(\textbf{1}_N \otimes x^*) = 0
\end{align}
\begin{asmp} \label{asmp:ExtendedPseudoGrad}
$\textbf F$ is Lipschitz continuous, $||\textbf{F}(\textbf{x})-\textbf{F}(\textbf{x}')|| \leq \theta ||\textbf{x}-\textbf{x}'||$, for all $\textbf{x},\textbf{x}' \in \reals^{Nn}$, for some $\theta >0$.
\end{asmp} 
Under Assumptions  \ref{asmp:Jsmooth}- \ref{asmp:ExtendedPseudoGrad},  if $\mu(\lambda_2(L)-\theta)>\theta^2$,  the unique NE, $x=x^*$, is globally asymptotically stable for 
all networked interconnected $\Sigma_i$, \eqref{eq:GradDynPartialInfo}, (cf. Theorem 1, \cite{gp18}). 

\vspace{-0.26cm}
\section{Problem Formulation} \label{sec:problem}
In this paper, we consider the problem of NE seeking for multi-integrator agents  in the presence of additive disturbance signals. The dynamics of each agent can be modelled by the following linear system, of order $r_i\geq1$
\vspace{-0.26cm}\begin{align} \label{eq:multiIntegrator}
x_i^{(r_i)} = u_i+d_i, \quad \forall i \in \mathcal{I}
\end{align}
where $x_i^{(r_i)}:=\frac{d^{r_i}x_i(t)}{dt^{r_i}}$. Each agent has a cost function $J_i(x_i,x_{-i})$ that it seeks to minimize. Agent $i$ is affected by disturbance, $d_i$, which can be modelled as being generated by 
\vspace{-0.26cm}\begin{align} \label{eq:Case1a_w}
\mathcal{D}_i:&\begin{cases}
\dot w_i = S_iw_i, \quad w_i(0) \in \mathcal{W}_i, \, \, \, \forall i \in \mathcal{I}\\
d_i = D_iw_i
\end{cases}
\end{align} 
where $w_i \in\reals^{q_i}$, $d_i  \in \reals^{n_i}$. We assume that $ \mathcal{W}_i$ a compact subset set of $\reals^{q_i}$ and  that $\mathcal{D}_i$ is marginally (neutrally) stable and observable, \cite{ib90}. Let $\mathcal{W}= \mathcal{W}_1 \times \dots \times \mathcal{W}_N$.
This setting is motivated by cases where the agents in the game have inherent dynamics. This occurs, for example, in the case of a game played between a network of velocity-actuated (single-integrator) or force-actuated (double-integrator) robots whose costs depend upon their position only. The disturbances are the result of a deterministic effect from the physical nature of the systems, e.g., wind pushing the mobile robots. These disturbances have known form (e.g., constant) but unknown parameters (e.g., strength). Therefore, 
we assume that each agent knows its $S_i$ and $D_i$ only, but has no knowledge of the initial condition $w_i(0)\in\mathcal{W}_i$ or the resulting solution $w_i(t)$. These are standard assumptions in the output regulation literature. Now the problem becomes one of finding control inputs, $u_i$,  that minimize the cost function $J_i(x_i,x_{-i})$ while  simultaneously rejecting the disturbances,  i.e., designing  dynamics $\Sigma_i$, under which the NE $x^*$ is asymptotically stable for the closed-loop irrespective of disturbances  (Fig. \ref{fig:Case1NoDyn}). 
We consider separately the single-integrator agents (Section \ref{sec:SingleIntegrator}) and double-integrator agents (Section \ref{sec:second_integrator}), and indicate how to extend the results  to multi-integrator agents. 


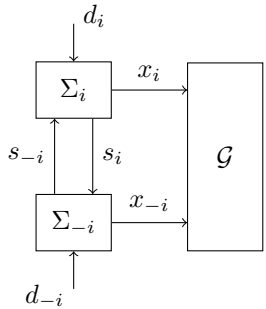
\begin{figure}[h!]
  \centering
\begin{tikzpicture}[auto, node distance=1cm]
    \node [input, name=input] {};
    \node [block, right = of input, minimum width = 1 cm, minimum height = 0.75 cm] (player) {$ \Sigma_i $};
    \node [input, name=dist, above= 0.5 cm of player] {};

    \node [block, below right = -0.75 cm and 1 cm of player, minimum height = 2.5 cm, minimum width = 1 cm] (game) {$ \mathcal{G} $};
    
    \node [block, below = of player, minimum width = 1 cm, minimum height = 0.75 cm] (others) {$\Sigma_{-i}$};
    \node [input, name=dist2, below= 0.5 cm of others] {};

\path (player.south) -- (player.south west) coordinate[pos=0.5] (a1);
\path (others.north) -- (others.north west) coordinate[pos=0.5] (b1);

\path (player.south) -- (player.south east) coordinate[pos=0.5] (a2);
\path (others.north) -- (others.north east) coordinate[pos=0.5] (b2);


  \draw [->] (dist) -- node [pos=-0.2] {$d_i$} (player);
 \draw [->] (dist2) -- node [pos=-0.2] {$d_{-i}$} (others);
    \draw [->] (a2) -- node [name=si] {$s_i$}(b2);
    \draw [->] (b1) -- node [name=so] {$s_{-i}$}(a1);

    \draw [->] (player) -- node [name=xi] {$x_i$}(player-|game.west);
    \draw [->] (others) -- node [name=xo] {$x_{-i}$}(others-|game.west);
\end{tikzpicture}
	\caption{Game with disturbances on the dynamics of each agent}
	\label{fig:Case1NoDyn}
\end{figure}

In each case we consider 
a partial-decision information setting, under local knowledge and communication over a graph $G_c$. We will show that if each player uses a gradient-play dynamics combined with an internal-model correction term that implements a reduced order observer for $w_i$, \cite{ai95}, (and a consensus-based dynamics), then every solution of the stacked dynamics of all agents stays bounded and will converge to the NE, $x^*$, irrespective of disturbances $w \in \mathcal{W}$.

\section{NE Seeking for Single-Integrator Agents} \label{sec:SingleIntegrator}
In this section, we consider a game $\mathcal{G}$ where each agent is modelled as 
\vspace{-0.26cm}\begin{align} \label{eq:Case1a}
\dot x_i = u_i+d_i, \quad \forall i \in \mathcal{I}
\end{align}
where  $d_i$ is generated by \eqref{eq:Case1a_w}, as in the example of a network of velocity actuated robots, and has a cost  $J_i(x_i,x_{-i})$, which it seeks to minimize while rejecting disturbances. 
We consider that each agent has partial (networked) information from his neighbours over graph, $G_c$.
Under Assumptions \ref{asmp:Jsmooth} and \ref{asmp:PseudoGrad}, the game has a unique NE. Inspired by \eqref{eq:GradDynPartialInfo}, our proposed $u_i$ is dynamic and is generated by 
\vspace{-0.2cm}\begin{align} \label{eq:AgentFullInfoInput}
\begin{split}
\dot{\textbf{x}}_{-i}^i &= -\mathcal{S}_i \sum_{j \in \mathcal{N}_i} (\textbf x^i - \textbf x^j)\\
\dot \xi_i &= S_i(K_ix_i+\xi_i) +K_i \nabla_i J_i(x_i,\textbf x_{-i}^i)\\
&\quad +K_i \mathcal R_i \sum_{j \in \mathcal{N}_i} (\textbf x^i - \textbf x^j)\\
u_i &= -\nabla_i J_i(x_i,\textbf{x}^i_{-i})- \mathcal R_i \sum_{j \in \mathcal{N}_i} (\textbf x^i - \textbf x^j)\\
&\quad-D_i(K_ix_i+\xi_i), \quad \quad \quad \quad \quad \forall i \in \mathcal{I}
\end{split}
\end{align}
where $K_i$ is chosen such that 
$\sigma(S_i-K_iD_i)\subset\mathbb{C}^-$. Note that \eqref{eq:AgentFullInfoInput} has a gradient-play term (evaluated at estimates) as well as a dynamic component $\dot{\xi}_i$ to reject disturbances,  
combined with a dynamic Laplacian-based estimate-consensus component $\dot{\textbf{x}}_{-i}^i $, which in steady-state should bring all $\textbf{x}^i$ to the consensus subspace, $\textbf{x}^i=\textbf{x}^j$. This leads to $\Sigma_i$ given by,
\vspace{-0.26cm}\begin{align} \label{eq:AgentPartialInfo}
\Sigma_i:\begin{cases}
\dot x_i &= -\nabla_i J_i(x_i,\textbf x_{-i}^i) - \mathcal R_i \sum_{j \in \mathcal{N}_i} (\textbf x^i - \textbf x^j)\\
&\quad -D_i(K_ix_i+\xi_i) + d_i \\ 
\dot{\textbf{x}}_{-i}^i &= -\mathcal{S}_i \sum_{j \in \mathcal{N}_i} (\textbf x^i - \textbf x^j)\\
\dot \xi_i &= S_i(K_ix_i+\xi_i) +K_i \nabla_i J_i(x_i,\textbf x_{-i}^i)\\
&\quad +K_i \mathcal R_i \sum_{j \in \mathcal{N}_i} (\textbf x^i - \textbf x^j),\quad \quad \forall i \in \mathcal{I}
\end{cases}
\end{align}
Compared to \eqref{eq:GradDynPartialInfo}, \eqref{eq:AgentPartialInfo} has an extra-component, $\xi_i$, that acts an internal model for the disturbance. The following result shows convergence to the NE irrespective of disturbances.  
\begin{thm} \label{thm:SingleIntPartial}
Consider a game $\mathcal{G}(\mathcal{I},J_i,\Omega_i)$ with partial information over a graph $G_c$ with Laplacian $L$ and 
 agent learning dynamics $\Sigma_i$, \eqref{eq:AgentPartialInfo}, where disturbance  $d_i$ is as in \eqref{eq:Case1a_w}. Under Assumptions  \ref{asmp:Jsmooth}, \ref{asmp:PseudoGrad}, \ref{asmp:GraphConnected} and \ref{asmp:ExtendedPseudoGrad},  if $\mu(\lambda_2(L)-\theta)>\theta^2$, 
then $\bar{\textbf x} = \textbf 1_N \otimes x^*$, where $x^*$ is the unique NE, is globally asymptotically stable for all networked interconnected $\Sigma_i$s, \eqref{eq:AgentPartialInfo}, for all $w \in \mathcal{W}$. Moreover, all players' estimates converge globally  to $\bar{\textbf x} = \textbf 1_N \otimes x^*$, for all $w \in \mathcal{W}$.

\end{thm}

\begin{proof} The idea of the proof is to express all agents' interconnected dynamics as a closed-loop dynamical system for which the NE is shown to be globally asymptotically stable irrespective of disturbances. To show stability we use  a suitable change of coordinates to put the system  in cascade form. Then we  exploit ISS properties induced by strong monotonicity of the pseudo-gradient and Lipschitz continuity of the extended pseudo-gradient.
	
In stacked form, using $\textbf{F}$,  \eqref{eq:ExtPseudo},   \eqref{eq:Case1a_w}, all interconnected $\Sigma_i$, \eqref{eq:AgentPartialInfo}, of all agents $i \in \mathcal{I}$, can be written as a closed-loop system \vspace{-0.26cm}
\begin{align} \label{eq:sigma_part_distb}
&\quad \dot w = Sw \notag\\
\Sigma:&\begin{cases}
\dot x = -\textbf{F}(\textbf{x}) - \mathcal R \textbf{L} \textbf{x}-D(Kx+\xi) + Dw\\
\mathcal{S} \dot{\textbf{x}} = -\mathcal{S} \textbf{L} \textbf{x}\\
\dot \xi = S(Kx+\xi) +K \textbf{F}(\textbf{x})+K \mathcal R \textbf{L} \textbf{x}
\end{cases}
\end{align}
with   $\mathcal R \!\!=\! \!\blkdiag\!(\mathcal R_1,\dots\mathcal R_N)$, $\!\mathcal S \!= \!\blkdiag\!(\mathcal S_1\dots\!\mathcal S_N)$, $\!\textbf{L} \!\!= \!L \otimes I_N $,  $\!\textbf{x}\!=\col(\textbf{x}^1,\dots\textbf{x}^N) $,  $\!\col(\textbf{x}_{-1}^1,\dots,\textbf{x}_{-N}^N\!)\!=\!\mathcal{S}\textbf{x}\!$ by $\textbf{x}_{-i}^i \!=\! \mathcal{S}_i \textbf{x}^i$. 

Consider the coordinate transformation $\xi \mapsto \rho \!:=\! w\!-\!(Kx \!+ \!\xi)$, so that  
$
\dot \rho 
\!=\!(S\!-\!KD)\rho
$. 
Note that from  $\textbf{x}^i = \mathcal{R}_i^T x_i+\mathcal{S}_i^T \textbf{x}^i_{-i}$ it follows that   $\textbf{x} = \mathcal R^Tx+\mathcal{S}^T\mathcal{S}\textbf{x}$. 
Using  $\mathcal{R}^T\mathcal{R}\!+\!\mathcal{S}^T\mathcal{S} \!=\! I_{Nn}$, from \eqref{eq:r_i}, and the previous relations, it follows that in the new coordinates, the stacked-form dynamics \eqref{eq:sigma_part_distb} are given as, 
\vspace{-0.26cm}\begin{align}
\dot w &= Sw \nonumber \\
\begin{split} \label{eq:StackedDynamicsDist}
\dot{\textbf{x}} &= -\mathcal{R}^T\textbf{F}(\textbf{x}) - \textbf{L} \textbf{x}+\mathcal{R}^TD\rho\\
\dot \rho &= (S-KD)\rho
\end{split}
\end{align}
We note that \eqref{eq:StackedDynamicsDist} is in cascade form from $\rho$ to $\textbf{x}$. By shifting the coordinates $\textbf{x} \mapsto \tilde{\textbf{x}} := \textbf x - \bar{\textbf{x}}$, where $\bar{\textbf{x}} = \textbf{1}_N \otimes x^*$, the dynamics of the $(\tilde {\textbf{x}},\rho)$ subsystem become
\vspace{-0.26cm}\begin{align} \label{eq:CascadeSystem2}
\begin{split}
\dot{\tilde{\textbf{x}}} &= -\mathcal R^T \textbf{F}(\tilde{\textbf{x}}+\bar{\textbf{x}})-\textbf{L}(\tilde{\textbf{x}}+\bar{\textbf{x}})+\mathcal R^TD\rho\\
\dot \rho &= (S-KD)\rho
\end{split}
\end{align}
Note that   \eqref{eq:CascadeSystem2} is again in cascade form, with the $\rho$-subsystem 
generating  the external input  for the $\tilde{\textbf{x}}$-subsystem. 
Consider 
$V(\tilde{\textbf{x}}) = \frac{1}{2} \|\tilde{\textbf{x}}\|^2$.   
Then, along solutions of  the $\tilde{\textbf{x}}$-subsystem  in \eqref{eq:CascadeSystem2}, using $\textbf{L}\!\bar{\textbf{x}}=  \textbf{0}_{Nn}$, it holds that \vspace{-0.2cm}
\begin{align}\label{bau_0}
\dot V &= -\tilde{\textbf{x}}^T \big ( \mathcal R^T[\textbf{F}(\tilde{\textbf{x}}+\bar{\textbf{x}}) + \textbf{L}\tilde{\textbf{x}} \! - 
\mathcal R^TD\rho \big )
\end{align}

Decompose $\reals^{Nn}$ as $\reals^{Nn}\!=\!C^n \!\oplus \!E^n$, where $C^n \!=\! \{\textbf{1}_N \otimes \!x\! |\ x \!\in\! \reals^n\}$ is the consensus subspace, and $E^n$ is its orthogonal complement.  Any $\textbf{x}\!\in \!\reals^{Nn}$ can be written as $\textbf{x \!}=\! \textbf{x}^\bot \!+\!\textbf{x}^{\|}$, 
where  $\textbf{x}^{\|} \!= \!P_C \textbf{x} \!\in \!C^n$, $\textbf{x}^\bot \! = \! P_E \textbf{x} \! \in \! E^n$, for 
$P_C \!=\! \frac{1}{N} \textbf{1}_N \! \otimes \!\textbf{1}_N^T \!\otimes \! I_n$, $ P_E  \!=\! I_{Nn} \!-\!\frac{1}{N}\textbf{1}_N \!\otimes \! \textbf{1}_N^T \!\otimes  \!I_n
$. Thus,  ${\textbf{x}}^{\|} \!=\! \textbf{1}_N \!\otimes \!x$, for some $x \! \in \! \reals^n$, and  
 $ \tilde{\textbf{x}} \!=\!\textbf x \! -\! \bar{\textbf{x}} \!=\!  \tilde{\textbf{x}}^\bot \!+\!\tilde{\textbf{x}}^{\|}$, 
where  $\tilde{\textbf{x}}^{\|}\! =\! \textbf{1}_N  \!\otimes \! (x\!-\!x^*)$, $\tilde{\textbf{x}}^\bot \!=\!{\textbf{x}}^\bot$.  Using $\textbf{F}(\bar{\textbf{x}}) \!=\!\textbf{0}_{n}$ by (\ref{eq:ExtPseudoZero}), from \eqref{bau_0} we get \vspace{-0.2cm}
\begin{align}\label{bau_00}
\begin{split}
\dot V &= -(\tilde{\textbf{x}}^\bot+\tilde{\textbf{x}}^{\|})^T \mathcal R^T[\textbf{F}(\tilde{\textbf{x}}+\bar{\textbf{x}}) - \textbf{F}(\bar{\textbf{x}})]\\
& -(\tilde{\textbf{x}}^\bot\!+\!\tilde{\textbf{x}}^{\|})^T\textbf{L}(\tilde{\textbf{x}}^\bot \!+\! \tilde{\textbf{x}}^{\|})\!+\!(\tilde{\textbf{x}}^\bot\!+\!\tilde{\textbf{x}}^{\|})^T\mathcal R^TD\rho
\end{split}
\end{align}
Note that $\textbf{L}\tilde{\textbf{x}}^{\|} \!= \!\textbf{0}_{Nn}$ and $\lambda_2(L)\|\tilde{\textbf{x}}^\bot\|^2 \!\leq \! \tilde{\textbf{x}}^\bot \textbf{L} \tilde{\textbf{x}}^\bot$,   $\forall \tilde{\textbf{x}}^\bot \! \in \! E^n$ by properties of the Laplacian under Assumption \ref{asmp:GraphConnected}. 
Adding and subtracting $\! \textbf{F}(\tilde{\textbf{x}}^{\|}\!+\!\bar{\textbf{x}}) $ in \eqref{bau_00}, with 
$\textbf{F}(\tilde{\textbf{x}}^{\|}\!+\!\bar{\textbf{x}}) \!=\! \textbf{F}(\textbf{1}_N \otimes x) \!=\! F(x)$, $\textbf{F}(\bar{\textbf{x}}) \!= \!\textbf{F}(\textbf{1}_N \otimes x^*) \!=\! F(x^*)$, and using $\mathcal R(\tilde{\textbf{x}}^{\|})\! =\!x-x^*$,  
 yields  \vspace{-0.25cm}
\begin{align*}
\begin{split}
\dot V &\leq -(\tilde{\textbf{x}}^\bot)^T\mathcal R^T[ \textbf{F}(\tilde{\textbf{x}}^\bot+\tilde{\textbf{x}}^{\|}+\bar{\textbf{x}}) - \textbf{F}(\tilde{\textbf{x}}^{\|}+\bar{\textbf{x}})]\\
&\quad-(\tilde{\textbf{x}}^\bot)^T\mathcal R^T[F(x) - F(x^*)]-\lambda_2(L)\|\tilde{\textbf{x}}^\bot\|^2\\
&\quad -(x-x^*)^T[ \textbf{F}(\tilde{\textbf{x}}^\bot+\tilde{\textbf{x}}^{\|}+\bar{\textbf{x}}) - \textbf{F}(\tilde{\textbf{x}}^{\|}+\bar{\textbf{x}})]\\
&\quad -(x-x^*)^T[F(x) - F(x^*)]+(\tilde{\textbf{x}}^\bot+\tilde{\textbf{x}}^{\|})^T\mathcal R^TD\rho
\end{split}
\end{align*}
Using $\|\textbf{F}(\!\tilde{\textbf{x}}^\bot\!+\!\tilde{\textbf{x}}^{\|}\!+\!\bar{\textbf{x}})\!-\!\textbf{F}(\!\tilde{\textbf{x}}^\|\!+\!\bar{\textbf {x}})\| \!\leq \!\theta \| \!\tilde{\textbf{x}}^\bot\|$ by Assumption \ref{asmp:ExtendedPseudoGrad}, $\|\mathcal R \!\tilde{\textbf{x}}^\bot\| \! \leq \! \|\mathcal R\| \|\tilde{\textbf{x}}^\bot\|$, $\|F(x)-F(x^*)\|\! \leq \!\bar \theta \|x-x^*\|\! \leq  \!\theta \|x-x^*\|$, $(x-x^*)^T[F(x)-F(x^*)] \! \geq \!\mu \|x-x^*\|^2$ by Assumption \ref{asmp:PseudoGrad}, yields
\begin{align*}
\begin{split}
\dot V &\leq \theta \|\tilde{\textbf{x}}^\bot \|^2 + \theta\|\tilde{\textbf{x}}^\bot\|\|x-x^*\|-\lambda_2(L)\|\tilde{\textbf{x}}^\bot\|^2\\
&\quad+\theta\|x-x^*\|\|\tilde{\textbf{x}}^\bot\|-\mu\|x-x^*\|^2+(\tilde{\textbf{x}}^\bot+\tilde{\textbf{x}}^\|)^T\mathcal R^TD\rho
\end{split}
\end{align*}
Using $\|x-x^*\| = \frac{1}{\sqrt{N}}\|\tilde{\textbf{x}}^\|\|$, we can write \vspace{-0.26cm}
\begin{align*}
\begin{split}
\dot V &\leq -\begin{bmatrix} \|\tilde{\textbf{x}}^\|\| & \|\tilde{\textbf{x}}^\bot\| \end{bmatrix} \begin{bmatrix} \frac{1}{N} \mu & -\frac{1}{\sqrt{N}} \theta \\ -\frac{1}{\sqrt{N}} \theta & \lambda_2(L)-\theta \end{bmatrix}\begin{bmatrix} \|\tilde{\textbf{x}}^\|\| \\ \|\tilde{\textbf{x}}^\bot\| \end{bmatrix}\\
&\quad +\|\tilde{\textbf{x}}^\bot+\tilde{\textbf{x}}^\|\|\|\mathcal R^TD\|\|\rho\|
\end{split}
\end{align*}
Then, given any $a>0$, for any $\| \tilde{\textbf{x}}^\bot+\tilde{\textbf{x}}^\| \| \geq \frac{\|\mathcal R^TD\|}{a}\|\rho\|$, \vspace{-0.26cm} 
\begin{align*}
\begin{split}
\dot V &\leq -\begin{bmatrix} \|\tilde{\textbf{x}}^\|\| & \|\tilde{\textbf{x}}^\bot\| \end{bmatrix} \begin{bmatrix} \frac{1}{N} \mu & -\frac{1}{\sqrt{N}} \theta \\ -\frac{1}{\sqrt{N}} \theta & \lambda_2(L)-\theta \end{bmatrix}\begin{bmatrix} \|\tilde{\textbf{x}}^\|\| \\ \|\tilde{\textbf{x}}^\bot\| \end{bmatrix}\\
&\quad +a\|\tilde{\textbf{x}}^\bot+\tilde{\textbf{x}}^\|\|^2
\end{split}
\end{align*}
Note that 
$\|\tilde{\textbf{x}}^\bot+\tilde{\textbf{x}}^{\|}\|^2=\|\tilde{\textbf{x}}^\bot\|^2+\|\tilde{\textbf{x}}^{\|}\|^2 =
\|\tilde{\textbf{x}}\|^2   
$, 
so that, for any $\| \tilde{\textbf{x}}^\bot+\tilde{\textbf{x}}^\| \| \geq \frac{\|\mathcal R^TD\|}{a}\|\rho\|$ we can write, \vspace{-0.2cm}
\begin{align}\label{bau_2}
\dot V \!\leq \!-\!\big [ \|\tilde{\textbf{x}}^\|\| \, \,  \|\tilde{\textbf{x}}^\bot\| \big ] \!\!\begin{bmatrix} \frac{1}{N} \mu - a &\! -\frac{1}{\sqrt{N}} \theta \\ -\frac{1}{\sqrt{N}} \theta &\! \lambda_2(L)-\theta - a \end{bmatrix}\!\!\begin{bmatrix} \|\tilde{\textbf{x}}^\|\| \\ \|\tilde{\textbf{x}}^\bot\| \end{bmatrix}
\end{align}
For the $\tilde{\textbf{x}}$-subsystem  in \eqref{eq:CascadeSystem2}  to be ISS, we need the matrix on the right-hand side to be positive definite. This holds for any $a\!>\!0$ such that $a\!<\!\frac{1}{N} \mu$, $a\!<\!\lambda_2(L) \!-\! \theta$, and $(\frac{1}{N} \mu \!-\!a)(\lambda_2(L) \!-\! \theta \!-\! a) \!-\! \frac{1}{N} \theta^2 \!>\! 0$. 
Since $\mu(\lambda_2(L)-\theta)>\theta^2$,  the intersection of the above inequalities is guaranteed to be nonempty and the matrix is positive definite for any such $a$. Then, for any such $a$,   $\dot V(\tilde{\textbf{x}}) \leq -W(\tilde{\textbf{x}}),\, \forall \|\tilde{\textbf{x}}\| \geq \frac{\|\mathcal R^TD\|}{a}\|\rho\|$, where $W(\tilde{\textbf{x}})$ is a positive definite function, hence the $\tilde{\textbf{x}}$-subsystem  in \eqref{eq:CascadeSystem2}   is ISS with respect to $\rho$ by Theorem \ref{thm:ISSLyapunov}. Since  $\!\dot \rho \!=\! (S\!-\!KD)\rho$ is asymptotically stable by (\ref{eq:AgentPartialInfo}), it follows that the origin of (\ref{eq:CascadeSystem2}) is asymptotically stable by Lemma \ref{lemma:ISSCascade}, hence $( \mathbf{1}_N \! \otimes  x^*,0)$   
is asymptotically stable for (\ref{eq:StackedDynamicsDist}), for any $w \in \mathcal{W}$. 
\end{proof}
\begin{remark}
Local results follow if Assumption \ref{asmp:ExtendedPseudoGrad} holds  only locally around $\mathbf{x}^*=\mathbf{1}_N \otimes x^*$. 
We note that the class of quadratic games satisfies Assumption \ref{asmp:ExtendedPseudoGrad} globally.
\end{remark}
\begin{remark}
In the special case of full-information, there is no need for estimates and the agent (closed-loop) learning dynamics $\Sigma_i$, \eqref{eq:AgentPartialInfo}, reduce to,
\vspace{-0.2cm}\begin{align} \label{eq:AgentFullInfo}
\Sigma_i:
\begin{cases}
\dot x_i = -\nabla_i J_i(x_i,x_{-i})-D_i(K_ix_i+\xi_i)+d_i, \\ 
\dot \xi_i = S_i(K_ix_i+\xi_i)+K_i\nabla_i J_i(x_i,x_{-i}), 
\end{cases}
\end{align}
The convergence result as in Theorem \ref{thm:SingleIntPartial} holds without the need for Assumptions \ref{asmp:GraphConnected} and \ref{asmp:ExtendedPseudoGrad}.
\end{remark}

\section{NE Seeking For Double Integrators} \label{sec:second_integrator}


In this section, we consider NE seeking for double-integrator agents with disturbances. Our motivation  is two-fold. 
Firstly, the agents in the game might have some sort of inherent dynamics, such as double integrator robots playing a game wherein the cost functions are functions of their positions. Each agent, therefore, cannot directly update its action, $x_i$, via choice of input $u_i$ and must take into account its inherent dynamics. Secondly, we may want to consider higher-order dynamics for learning, 
as done extensively in the optimization literature, e.g., the heavy-ball method. 



Consider that each agent is modelled as a double integrator
\begin{align} \label{eq:DoubleIntegratorDist}
\begin{cases}
\dot x_i = v_i\\
\dot v_i = u_i+d_i
\end{cases}
\end{align}
where 
$x_i, v_i, u_i,d_i\in \reals^{n_i}$, $d_i$ generated by \eqref{eq:Case1a_w}. Each agent minimizes its cost function  $J_i(x_i,x_{-i})$, 
with the constraint that its steady-state velocity is zero. This setting is motivated for example in the case of a network of mobile, force-actuated robots whose costs depend on their positions only. At steady state, necessarily, their velocities must be zero. This requirement can be seen as the result of a quadratic penalty term, $J_{v_i}(v_i)\!=\! \frac{1}{2}\|v_i\|^2$, on the velocity of each agent. Thus, the overall cost function for each agent is given by $\bar J(x_i,x_{-i},v_i) \!=\!  J(x_i,x_{-i}) \!+\! \frac{1}{2}\|v_i\|^2$ and the resulting NE is
\vspace{-0.15cm}   
\begin{align} \label{eq:doubleNESet}
\!\!\!\!\Gamma_{NE}\!= \big \{\!(x,\!v\!) \!\in \! \reals^n \!\!\times \! \!\reals^n | \!\nabla_i J_i(x_i,x_{-i}) \!=\! 0,v_i\!=\!0, \forall i \! \in\! \mathcal{I}\big \}
\end{align}
Under Assumptions \ref{asmp:Jsmooth} and \ref{asmp:PseudoGrad}, $x^*$ is unique. 
By \eqref{eq:doubleNESet},  $(x^*,v^*)$ is such that \vspace{-0.2cm}
\begin{align}
\label{eq:PsudoGradZeroDouble}
F(x^*) &= 0, \quad 
v^* = 0
\end{align}
We consider partial-information learning dynamics under which the NE of the game is reached in the presence of additive disturbances. 

We propose the following dynamic feedback \vspace{-0.2cm}
	\begin{align} \label{eq:doubleIntLearningDist}
	\begin{split}
	\dot{\boldsymbol{\gamma}}_{-i}^i \!\!\!&= -\mathcal{S}_i\sum_{j \in \mathcal{N}_i}(\boldsymbol{\gamma}^i - \boldsymbol{\gamma}^j)\\
	\!\dot \xi_i \!\!\!&= S_i(K_iv_i+\xi_i)+K_i \nabla_i J_i(x_i+b_i v_i,\boldsymbol{\gamma}_{-i}^i)\\
	&\quad+K_i \Big(\frac{1}{b_i}v_i + \mathcal{R}_i\sum_{j \in \mathcal{N}_i}(\boldsymbol{\gamma}^i - \boldsymbol{\gamma}^j)\Big)\\
	\!\!u_i \! \!&=\! -\! \nabla_iJ_i(x_i + b_iv_i,\boldsymbol{\gamma}_{-i})- \frac{1}{b_i}v_i\\
	&\quad \!-\! \mathcal{R}_i \!\sum_{j \in \mathcal{N}_i}\!(\boldsymbol{\gamma}^i \! -\!\! \boldsymbol{\gamma}^j)\!-\!D_i(K_iv_i\!+\!\xi_i)\!
	\end{split}
	\end{align}
	where $K_i$ is such that $\sigma(S_i\!-\!K_iD_i) \! \!\subset \!\mathbb{C}^-$ and $b_i\!>\!0$ and $\boldsymbol{\gamma}^i$ is agent $i$'s estimate variable.  Note that in $\nabla_iJ_i$ is evaluated at a predicted point, $x_i+b_iv_i$. We denote $\boldsymbol{\gamma}_j^i$ agent $i$'s prediction of $x_j+b_jv_j$.
	\eqref{eq:doubleIntLearningDist} uses a similar internal model and Laplacian-based consensus scheme \eqref{eq:AgentFullInfoInput}, as in Section \ref{sec:SingleIntegrator}. However, instead of each agent estimating the others' actions, $x_{-i}$, they estimate the predicted actions $\{x_j + b_jv_j|j\in \mathcal{I},j\neq i\}$ 
	used to  evaluate the gradient at the predicted point.  Each agent will then share these estimates as well as their own prediction with their neighbours. Therefore, each agent $i$ computes $\boldsymbol{\gamma}_{-i}^i=\col(\boldsymbol{\gamma}_{1}^i,\dots,\boldsymbol{\gamma}_{i-1}^i,\boldsymbol{\gamma}_{i+1}^i,\dots,\boldsymbol{\gamma}_{N}^i) \in \reals^{n_{-i}}$ and uses these estimates when evaluating its gradient, $\nabla_i J_i(x_i+b_iv_i,\boldsymbol{\gamma}_{-i}^i)$. Intuitively, each agent makes a prediction on the future state of the game, $\!x_i\!+\!b_iv_i$, based on the current actions and velocities, and evaluates its gradient with respect to $x_i$ at this point. We denote $\boldsymbol{\gamma}^i=\col(\boldsymbol{\gamma}_{1}^i,\dots,\boldsymbol{\gamma}_{i-1}^i,x_i+b_iv_i,\boldsymbol{\gamma}_{i+1}^i,\dots,\boldsymbol{\gamma}_{N}^i) \in \reals^n$ and $\boldsymbol{\gamma}=\col(\boldsymbol{\gamma}^1,\dots,\boldsymbol{\gamma}^N) \in \reals^{Nn}$. 
\begin{remark}\label{rem:rem1}
From an agent perspective, the intuition behind \eqref{eq:doubleIntLearningDist} 
is that  
each agent evaluates its partial-gradient at a predicted future point, $x_i+b_iv_i$, obtained as a first-order prediction from the current action and velocity of each agent, with  a negative feedback on its velocity. This can be viewed as resulting from the quadratic penalty term associated with the velocity of each agent. 
In addition, consider the disturbance-free case and recall that gradient-play is a method that works well for single-integrators, i.e., systems with unit relative degree. By creating a fictitious output $\boldsymbol{\gamma}^i_i\!:\!=\!x_i\!+\!b_iv_i$ we decrease the relative degree of each agent to $\{1,\dots,1\}$. This creates a hyperplane  $x+\mathcal{B}v -x^* = 0$, where $\mathcal{B}=\blkdiag(b_1I_{n_1},\dots,b_NI_{n_N})$, on which the pseudo-gradient map is zero. The pseudo-gradient feedback makes this hyperplane attractive for the double-integrator system.  The feedback stabilizes $v=0$ and renders this hyperplane invariant, thereby stabilizing $x = x^*$. 

\end{remark}
\begin{remark}\label{rem:rem2}
	We note that \eqref{eq:doubleIntLearningDist} 
	is similar to a passivity-based group coordination design, e.g., \cite{ma07}. Indeed, the inner-loop feedback $u_i = -\frac{1}{b_i}v_i$ renders the agent dynamics passive with $\boldsymbol{\gamma}^i_i = x_i+b_iv_i$ as output. However, we stress that the feedback $\nabla_iJ_i(x_i+b_iv_i,\boldsymbol{\gamma}_{-i})$ is not necessarily the proper gradient of any function, as required in \cite{ma07}. Therefore, individually,  each agent is not a passive system  when the feedback is added, due to coupling to the others' actions via the cost function. This precludes using a passivity approach as in \cite{ma07}. Rather, here we use a combined ISS approach to deal with both the disturbance and the higher-order stabilization. 
\end{remark}

The choice of feedback \eqref{eq:doubleIntLearningDist} yields 
learning (closed-loop) dynamics given by  \vspace{-0.27cm}
\begin{align} \label{eq:doubleIntDynDistPart}
\!\Sigma_i\!:\!\begin{cases}\!
\dot{\boldsymbol{\gamma}}_{-i}^i \!\!\!&= -\mathcal{S}_i\sum_{j \in \mathcal{N}_i}(\boldsymbol{\gamma}^i - \boldsymbol{\gamma}^j),\quad \quad \quad \quad \forall i \in \mathcal{I}\\
\!\dot x_i \!\!\!&=v_i\\
\!\dot v_i \!\!\!&=\!-\! \nabla_iJ_i(x_i + b_iv_i,\boldsymbol{\gamma}_{-i})- \frac{1}{b_i}v_i\\
&\quad \!-\! \mathcal{R}_i \!\sum_{j \in \mathcal{N}_i}\!(\boldsymbol{\gamma}^i \! -\!\! \boldsymbol{\gamma}^j)\!-\!D_i(K_iv_i\!+\!\xi_i)\!+\!d_i\\
\!\dot \xi_i \!\!\!&= S_i(K_iv_i+\xi_i)+K_i \nabla_i J_i(x_i+b_i v_i,\boldsymbol{\gamma}_{-i}^i)\\
&\quad+K_i \Big(\frac{1}{b_i}v_i + \mathcal{R}_i\sum_{j \in \mathcal{N}_i}(\boldsymbol{\gamma}^i - \boldsymbol{\gamma}^j)\Big)\\   
\end{cases}
\end{align}
\begin{thm}\label{thm:second_partInfo_w_Disturb}
	Consider a game $\mathcal{G}(\mathcal{I},J_i,\reals^{n_i})$ with partial information over a graph $G_c$ with Laplacian $L$ and 
	learning dynamics $\Sigma_i$, \eqref{eq:doubleIntDynDistPart}, where disturbance  $d_i$ is generated by \eqref{eq:Case1a_w}. Under Assumptions  \ref{asmp:Jsmooth}-\ref{asmp:ExtendedPseudoGrad},  
	 if $\mu(\lambda_2(L)\!-\!\theta)\!>\!\theta^2$ then $1_N \otimes x^*$, where $x^*$ is the unique NE, is globally asymptotically stable for all networked interconnected $\Sigma_i$, for all $w \in \mathcal{W}$. Moreover, each player's estimates converge globally  to the NE value, $\bar{\boldsymbol{\gamma}} = \textbf 1_N \otimes x^*$. 
\end{thm}
\begin{proof}
	The idea of the proof is similar to that of Theorem \ref{thm:SingleIntPartial}. We use a change of coordinates to express the closed-loop dynamics in cascade form and use ISS arguments the show stability of the NE for the overall cascade system, irrespective of disturbance. The difference lies in the fact that \eqref{eq:doubleIntDynDistPart} has extra terms due to the higher-order dynamics $\dot v_i$ that must be incorporated into the cascade.
	
	The stacked dynamics of \eqref{eq:doubleIntDynDistPart} is given by \vspace{-0.3cm}
	\begin{align} \label{eq:doubleIntDynFullDistPart}
	\begin{split}
	&\quad\dot w \!=\! Sw\\
	\Sigma:&\begin{cases}
	\mathcal{S}\dot{\boldsymbol{\gamma}} \! = \!-\mathcal{S}\textbf{L}\boldsymbol{\gamma}\\
	\dot x \!=\!v\\
	\dot v \!= \!-\mathcal{B}^{-1}v\!-\!\textbf{F}(\boldsymbol{\gamma})\!-\!\mathcal{R}\textbf{L}\boldsymbol{\gamma}\!-\!D(Kv\!+\!\xi)\!+\!Dw\\
	\dot \xi \!= \!S(Kv+\xi)+K(\textbf{F}(\boldsymbol{\gamma})+\mathcal{B}^{-1}v+\mathcal{R}\textbf{L}\boldsymbol{\gamma})
	\end{cases}
	\end{split}
	\end{align}
	Note that $\mathcal{R} \boldsymbol{\gamma} = [\mathcal{R}_i \boldsymbol{\gamma}^i]_{i\in \mathcal{I}} =  
	 [x_i + b_i v_i]_{i \in \mathcal{I}} = x + \mathcal{B} v $. 
	Let the coordinate transformation $x \mapsto \mathcal{R}\boldsymbol{\gamma} \!: = \! x \!+\! \mathcal{B} v$. Then, 
	 \vspace{-0.25cm}
	\begin{align*}
	\mathcal{R}\dot{\boldsymbol{\gamma}} &= -\mathcal{B}\textbf{F}(\boldsymbol{\gamma})-\mathcal{B}\mathcal{R}\textbf{L}\boldsymbol{\gamma}-\mathcal{B}D(Kv+\xi-w).
	\end{align*}
Combining this with the second equation in \eqref{eq:doubleIntDynFullDistPart}, by using the properties of $\mathcal{R}$ and $\mathcal{S}$, $\mathcal{R}^T\mathcal{R}+ \mathcal{S}^T\mathcal{S} = I$, yields  that 	 \vspace{-0.25cm}
$$
	\!\!\dot{\boldsymbol{\gamma}} \!=\! -\mathcal{R}^T\mathcal{B}\textbf{F}(\!\boldsymbol{\gamma}\!)\!-\!(\mathcal{R}^T\mathcal{B}\mathcal{R}\!+\!\mathcal{S}^T\mathcal{S})\textbf{L}\boldsymbol{\gamma}
$$
	Let $\xi \mapsto \rho \! := \! w-(K v\!+\!\xi)$, so that 
	$\dot \rho =(S-KD)\rho$. 
	Consider also $\boldsymbol{\gamma} \mapsto \tilde{\boldsymbol{\gamma}} := \boldsymbol{\gamma}-\bar{\boldsymbol{\gamma}}$. Then, in the new coordinates, using $\textbf{L}\bar{\boldsymbol{\gamma}} = 0$, the dynamics of the $(\tilde{\boldsymbol{\gamma}},v,\rho)$ are given by \vspace{-0.2cm}
	\begin{align} 
	\label{eq:doubleIntDynStackedNew2bPartialDist}
	\dot v &= -\mathcal{B}^{-1} v - \textbf{F}(\tilde{\boldsymbol{\gamma}}+\bar{\boldsymbol{\gamma}}) - \mathcal{R}\textbf{L}\tilde{\boldsymbol{\gamma}}+D\rho\\
	\label{eq:doubleIntDynStackedNew1bPartialDist}
	\dot{\tilde{\boldsymbol{\gamma}}}\! &=\! -\mathcal{R}^T\mathcal{B}\textbf{F}(\tilde{\boldsymbol{\gamma}}\!+\!\bar{\boldsymbol{\gamma}})\!-\!(\mathcal{R}^T\mathcal{B}\mathcal{R}\!+\!\mathcal{S}^T\mathcal{S})\textbf{L}\tilde{\boldsymbol{\gamma}}\!+\!\mathcal{R}^T\mathcal{B}D\rho\\
	\label{eq:doubleIntDynStackedNew3bPartialDist}
	\dot \rho &= (S-KD)\rho
	\end{align}
Note that the (\ref{eq:doubleIntDynStackedNew2bPartialDist}-\ref{eq:doubleIntDynStackedNew3bPartialDist}) is in cascade form with subsystem $(\tilde{\boldsymbol{\gamma}}, \rho)$ (\ref{eq:doubleIntDynStackedNew1bPartialDist}, \ref{eq:doubleIntDynStackedNew3bPartialDist}) generating the external input for \eqref{eq:doubleIntDynStackedNew2bPartialDist}. In turn, (\ref{eq:doubleIntDynStackedNew1bPartialDist}, \ref{eq:doubleIntDynStackedNew3bPartialDist}) is in cascade form with (\ref{eq:doubleIntDynStackedNew3bPartialDist}) generating input $\rho$ for (\ref{eq:doubleIntDynStackedNew1bPartialDist}). We show first that (\ref{eq:doubleIntDynStackedNew1bPartialDist}) is ISS with respect to $\rho$. 
	Consider  \vspace{-0.25cm}
	\begin{align}\label{V_doubleInt_gamma}
	V(\tilde{\boldsymbol{\gamma}}) = \frac{1}{2} \tilde{\boldsymbol{\gamma}}^T \mathcal{R}^T\mathcal{B}^{-1}\mathcal{R}\tilde{\boldsymbol{\gamma}}+\frac{1}{2}\tilde{\boldsymbol{\gamma}}^T \mathcal{S}^T \mathcal{S}\tilde{\boldsymbol{\gamma}}
	\end{align}
which is positive definite. 	
	Taking the time-derivative of $V$ \eqref{V_doubleInt_gamma} along the solutions of \eqref{eq:doubleIntDynStackedNew1bPartialDist}, using $\textbf{F}(\bar{\boldsymbol{\gamma}})=0$, cf. (\ref{eq:ExtPseudoZero}), and properties of $\mathcal{R}$, $\mathcal{S}$, e.g. $\mathcal{R}\mathcal{S}^T=0$, $\mathcal{R}\mathcal{R}^T=I$, yields \vspace{-0.25cm}
	\begin{align*}
	\begin{split}
	\dot V \!&=\! -\tilde{\boldsymbol{\gamma}}^T\mathcal{R}^T\mathcal{B}^{-1}(\mathcal{B}\textbf{F}(\!\tilde{\boldsymbol{\gamma}}\!+\!\bar{\boldsymbol{\gamma}})\!+\!\mathcal{B}\mathcal{R}\textbf{L}\tilde{\boldsymbol{\gamma}}\!-\!\mathcal{B}D\rho)-\tilde{\boldsymbol{\gamma}}^T\mathcal{S}^T\mathcal{S}\textbf{L}\tilde{\boldsymbol{\gamma}}\\
	\!&=-\tilde{\boldsymbol{\gamma}}^T(\mathcal{R}^T\textbf{F}(\tilde{\boldsymbol{\gamma}}+\bar{\boldsymbol{\gamma}}) + \textbf{L}\tilde{\boldsymbol{\gamma}}-\mathcal{R}^TD\rho)
	\end{split}
	\end{align*}
which is similar to \eqref{bau_0}. 
Then, following an argument  as in the proof of Theorem \ref{thm:SingleIntPartial}, it follows that 
 the $\tilde{\boldsymbol{\gamma}}$ subsystem of \eqref{eq:doubleIntDynStackedNew1bPartialDist} is ISS with input $\rho$. Since the origin of the $\rho$ subsystem is globally asymptotically stable, then the origin of the $(\tilde{\boldsymbol{\gamma}},\rho)$ subsystem is globally asymptotically stable (cf.  Lemma \ref{lemma:ISSCascade}).
	
	Now consider the $v$-subsystem \eqref{eq:doubleIntDynStackedNew2bPartialDist} with input $(\tilde{\boldsymbol{\gamma}},\rho)$ and 
$	V_2(v) = \frac{1}{2}\|v\|^2
$. 
	Along \eqref{eq:doubleIntDynStackedNew2bPartialDist},  using  Assumption \ref{asmp:PseudoGrad},  \vspace{-0.2cm}%
	\begin{align*}
	\dot V_2 &= -v^T\mathcal{B}^{-1}(v)-v^T(\textbf{F}(\tilde{\boldsymbol{\gamma}}+\bar{\boldsymbol{\gamma}})+\mathcal{R}\textbf{L}\tilde{\boldsymbol{\gamma}}-D\rho)\\
	&\!\leq \!-\frac{1}{b_m}\|v\|^2\!+\!\|v\|\|\textbf{F}(\!\tilde{\boldsymbol{\gamma}}\!+\!\bar{\boldsymbol{\gamma}})\|
	+\|v\|\|\mathcal{R}\textbf{L}\tilde{\boldsymbol{\gamma}}\|+\|v\|\|D\|\|\rho\| \\
	&\leq -\frac{1}{b_m}\|v\|^2  + \beta\|v\|(\|\rho\|+\|\tilde{\boldsymbol{\gamma}}\|)
	\end{align*}
where  $b_m \!=\! \max_{i\in \mathcal{I}} b_i$, $\beta \! = \! \max\{\|\mathcal{R}\textbf{L}\|\!+\!\theta,\|D\|\}$.  
		Thus 
$
		\dot V_2 \leq  -\frac{1}{b_m}\|v\|^2  + \beta\|v\| \sqrt{2(\|\rho\|^2+\| \tilde{\boldsymbol{\gamma}}\|^2)}
$ 
	or,  \vspace{-0.26cm}
	\begin{align*}
	\dot V_2 \leq -\frac{1}{b_m}\|v\|^2 + \sqrt{2}\beta\|v\|\|\tilde{u}\|.
	\end{align*}
where $\tilde{u}:=\col(\tilde{\boldsymbol{\gamma}},\rho)$. Hence, $ 
	\dot V_2  \! 	\leq  \!-\big(\frac{1}{b_m} \!-\! b \big)\|v\|^2,\ \forall \|v\| \!\geq \!\frac{\sqrt{2}\beta}{b} \!\|\tilde{u}\|  
	$, 
	for  any $0\!<\!b\!<\!\frac{1}{b_m}$.
	Therefore, by Theorem \ref{thm:ISSLyapunov}, \eqref{eq:doubleIntDynStackedNew2bPartialDist} is ISS with $\tilde{u} = (\tilde{\boldsymbol{\gamma}},\rho)$. Since the origin of $(\tilde{\boldsymbol{\gamma}},\rho)$ subsystem is globally asymptotically stable, by Lemma \ref{lemma:ISSCascade}, the origin of \eqref{eq:doubleIntDynStackedNew1bPartialDist}-\eqref{eq:doubleIntDynStackedNew3bPartialDist} is globally asymptotically stable, hence  $(x^*,0)$ is globally asymptotically stable for \eqref{eq:doubleIntDynFullDistPart} for all $w \in \mathcal{W}$.
\end{proof}

\begin{remark}
	In the full information case, there is no need for estimate and \eqref{eq:doubleIntDynDistPart} reduces to \vspace{-0.22cm} 
	\begin{align}\label{eq:doubleIntDynDist}
	\!\Sigma_i\!\!:\!\! \begin{cases}
	\!\!\dot x_i \!\! \!\!\!&= \!v_i\\
	\!\!\dot v_i \!\!\!\!\!&= \!-\!\nabla_iJ_i(x_i\!\!+\!\!b_iv_i,\!x_{-i}\!\!+\!\!b_{-i}v_{-i})
	\!-\! \!\frac{1}{b_i}v_i \!-\!\!D_i(K_iv_i\!+\!\!\xi_i)\!+\!\!d_i\\
	\!\!\dot \xi_i \!\!\!\!\!&= \!S_i(K_iv_i\!+\!\xi)
	\!+\! K_i\big(\nabla_iJ_i(x_i\!+\!b_iv_i,\!x_{-i}\!+\!\!b_{-i}v_{-i})\!+\! \frac{1}{b_i}v_i\!\big)
	\end{cases}
	\end{align}
	The convergence results hold without the need for Assumptions \ref{asmp:GraphConnected} and \ref{asmp:ExtendedPseudoGrad}.
		Furthermore, the disturbance-free, 
		higher-order learning dynamics generated by \eqref{eq:doubleIntDynDist} 
		is  \vspace{-0.23cm}
		$$\ddot x + \mathcal{B}^{-1} \dot x + F(x + \mathcal{B} \dot x) =0$$
		which resembles heavy-ball with friction dynamics used in optimization, \cite{Attouch2000,ALvarez2000}. 
\end{remark}

\begin{remark}
The results  from this section 
can easily be extended to multi-integrator agents. 
Consider that  each agent is modelled as a $\!r_i^{th}\!$ order integrator, $r_i \! \geq \!2$, 
\vspace{-0.26cm}\begin{align*} 
\dot x_i &= C_i v_i\\
\dot v_i &= A_iv_i+B_i(u_i+d_i), \quad \forall i \in \mathcal{I}
\end{align*}
where 
$
A_i \!=\! \begin{bmatrix} \!0_{n_i(r_i-2)\!\times \!n_i} \!&\! I_{n_i(r_i-2)} \\ \!0_{n_i\!\times \! n_i}  \!&\! \!0_{n_i\! \times \! n_i(r_i-2)} \end{bmatrix}$, $ B_i \!=\! \begin{bmatrix} \!0_{n_i(r_i-2)\!\times \!n_i} \\ \! I_{n_i} \! \end{bmatrix}$, 
$C_i \!=\! \begin{bmatrix} \! I_{n_i} \! \!& \!0_{n_i \! \times \! n_i(r_i-2)} \!\end{bmatrix}$, 
$v_i = \col(v_i^1,\dots,v_i^{r_i-1})$,
 and has a cost function $J_i(x_i, x_{-i})$. In this case, 
$ \gamma_i \!:= \!x_i  \!+\! \begin{bmatrix} c_i^T \!\otimes \!I_{n_i} \!&\! I_{n_i} \end{bmatrix}  \!v_i$ where $c_i^T \!=\! \begin{bmatrix} c_{i,1} \!&\! \dots \!  & \!c_{i,(r_i-2)} \end{bmatrix}$, $c_{i,k}$ are  the coefficients of any $\!(r_i\!-\!1)^{\!th}\!$ order Hurwitz polynomial with $c_{i,0} \!\!=\! 1$, $c_{i, (r_i-1)} \!\!=\!1$, and 
$u_i \!:=\! -\nabla_i J_i(\gamma_i,\boldsymbol{\gamma}^i_{-i}) \!-\!  \begin{bmatrix} I_{n_i} \!& \! c_i^T \!\otimes \!I_{n_i}  \end{bmatrix} \! v_i$. 
When $r_i \!=\!2$, this feedback reduces to the one for the second-order integrator with $b_i=1$. Then a  dynamic learning scheme similar to \eqref{eq:doubleIntDynDistPart} can be developed, by appropriately augmenting with reduced-order observer for the disturbance, and consensus-dynamics for the estimates $\boldsymbol{\gamma}^i_{-i}$. 

The resulting agent learning dynamics  are given as
	\begin{align} \label{eq:multiIntDyn}
	\Sigma_i: \begin{cases}
	\dot{\boldsymbol{\gamma}}_{-i}^i &= -\mathcal{S}_i\sum_{j \in \mathcal{N}_i}(\boldsymbol{\gamma}^i - \boldsymbol{\gamma}^j)\\
	\dot x_i &= C_iv_i\\
	\dot v_i &= A_iv_i - B_i\Big(\nabla_iJ_i(\boldsymbol{\gamma}_i,\boldsymbol{\gamma}_{-i})+\begin{bmatrix} I_{n_i} \!& \! c_i^T \!\otimes \!I_{n_i}  \end{bmatrix} \! v_i\\
	&\quad+\mathcal{R}_i\sum_{j \in \mathcal{N}_i}(\boldsymbol{\gamma}^i - \boldsymbol{\gamma}^j)-D_iw\\
	&\quad+D_i(K_iv_i^{r_i-1}+\xi_i)\Big)\\
	\dot \xi_i &= S_i(K_iv_i^{r_i-1}+\xi_i)+K_i\Big(\nabla_i J_i(\boldsymbol{\gamma}_i^i,\boldsymbol{\gamma}_{-i}^i)\\
	&\quad+\mathcal{R}_i\sum_{j \in \mathcal{N}_i}(\boldsymbol{\gamma}^i - \boldsymbol{\gamma}^j)+\begin{bmatrix} I_{n_i} \!& \! c_i^T \!\otimes \!I_{n_i}  \end{bmatrix} \! v_i\Big)
	\end{cases}
	\end{align}
	which for $r_i=2$ reduces to \eqref{eq:doubleIntDynDistPart}.
	\begin{thm} \label{thm:MultiPartialDist}
		Consider a game $\mathcal{G}(\mathcal{I},J_i,\Omega_i)$ with partial information communicated over a graph $G_c$ with Laplacian $L$ and  agent dynamics given by $\Sigma_i$, \eqref{eq:multiIntDyn}. Under Assumptions  \ref{asmp:Jsmooth}, \ref{asmp:PseudoGrad}, \ref{asmp:GraphConnected} and \ref{asmp:ExtendedPseudoGrad},  if $\mu(\lambda_2(L)-\theta)>\theta^2$ then the unique NE, $x=x^*$, is globally asymptotically stable for \eqref{eq:multiIntDyn} for all $w \in \mathcal{W}$. Moreover, each player's estimates converge globally  to the NE values, $\bar{\textbf x} = \textbf 1_N \otimes x^*$, for all $w \in \mathcal{W}$.
\end{thm}
\begin{proof} Similar to Theorem \ref{thm:second_partInfo_w_Disturb}. 
\end{proof}
\end{remark}

\section{Numerical Examples}\label{sec:simulations}

In this  section we consider two application scenarios: an optical network OSNR game and a sensor network game. In both examples, our algorithms are compared with the full and partial-information gradient-play in the presence of disturbances.
\vspace{-0.3cm}
\subsection{OSNR Game}
%

Consider an optical signal-to-noise ratio (OSNR) model for  wavelength-division multiplexing (WDM) links \cite{lp12}, where 10 channels, $\mathcal{I} = \{1,\dots,10\}$, are transmitted over an optically amplified link. We consider each channel as an agent and denote each agent's transmitting power as $x_i$, while  the noise power of each channel as $n_i^0$. 
Each agent attempts to maximize its OSNR on its channel by adjusting its transmission power. Each agent has a cost function as in \cite{pp09}, given by \vspace{-0.2cm}
$$
\!\!J_i(x_i,x_{-i}) \!=\! a_ix_i\!+\!\frac{1}{P^0\!-\!\sum_{j \in \mathcal{I}}\!x_j}
\!-\!b_i\!\ln\!\Big(1\!+\!c_i \!\frac{x_i}{n_i^0\!+\!\!\sum_{ j \neq i }\!\Gamma_{ij} x_j}\Big)  
$$
where $a_i>0$ is a pricing parameter, $P^0$ is the total power target of the link, $b_i>0$, and $\Gamma = [\Gamma_{ij}]$ is the link system matrix, with parameters 
as in \cite{sp16}. Each channel (agent) has dynamics (\ref{eq:Case1a}), where  the disturbance is generated due to the pilot-tones used for network tracing and monitoring,  \cite{tp06}, which take the form of a sinusoidal signal with a unique frequency assigned for each channel and unknown modulation. Thus $d_i = P^0[1+m_i\sin(2\pi f_it)]$,  where $m_i = 0.1i$ (unknown modulation index) and frequency $f_i = 10i$ kHz,  $i \in \mathcal{I}$. 
\begin{figure}[h!]
 \vspace{-0.26cm}
 	\centering
	\includegraphics[trim=0cm 0cm 0cm 0cm,width=2.2in]{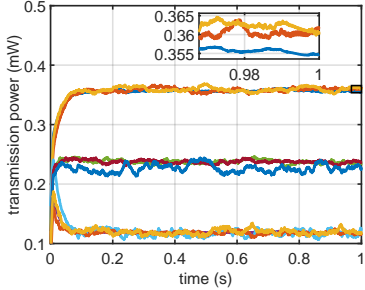}
	\caption{Gradient-play dynamics \eqref{eq:GradDyn} subject to disturbances}\label{fig:4b}
	 \vspace{-0.26cm}
\end{figure}
\begin{figure}[h!]
 \vspace{-0.26cm}
 	\centering
	\includegraphics[trim=0cm 0cm 0cm 0cm,width=2.2in]{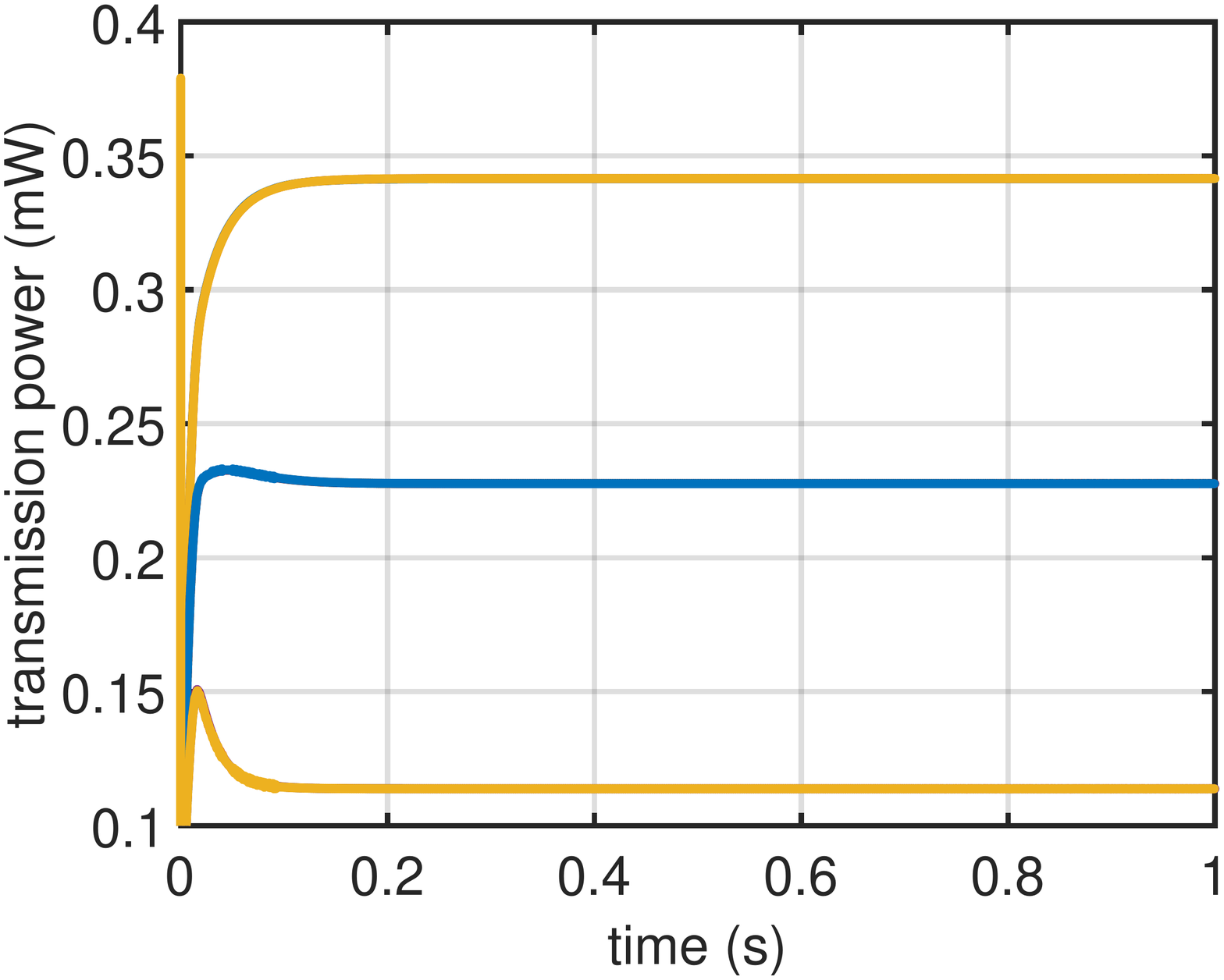}
	 \vspace{-0.26cm}
	\caption{Agent dynamics $\Sigma_i$ \eqref{eq:AgentFullInfo} subject to disturbances}\label{fig:4c}
\end{figure}
\begin{figure}[h!]
\vspace{-0.26cm}
\centering
\includegraphics[ width =4.5cm]{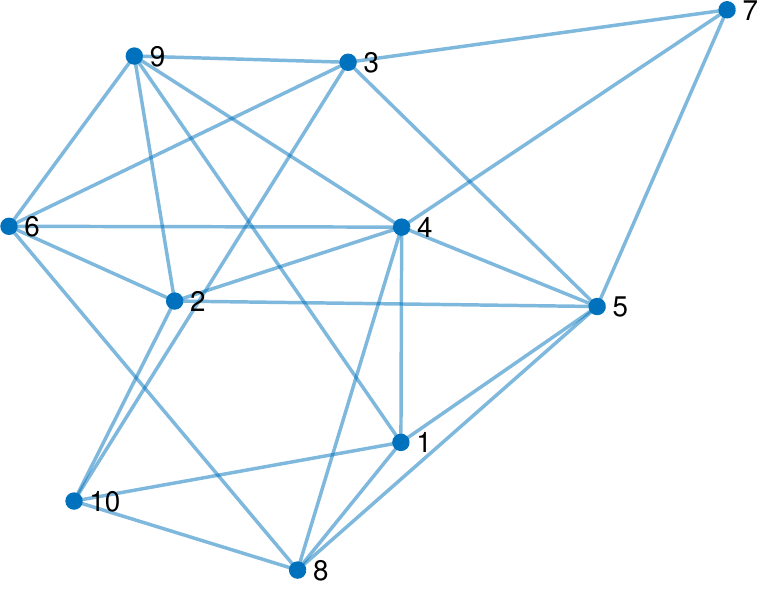}
\caption{Random communication graph, $G_c$, $\lambda_2 = 2.6158$}
\label{fig:GraphOSNR}
\end{figure}
First, we consider that each agent has full information about the others' actions and we compare the results of agent dynamics,  (\ref{eq:AgentFullInfo}),  with a standard gradient-play scheme  \eqref{eq:GradDyn}. As seen in  Fig. \ref{fig:4b} and Fig. \ref{fig:4c},  \eqref{eq:GradDyn} do not the reject disturbances (sustained fluctuations in the OSNR values), while   (\ref{eq:AgentFullInfo}) successfully reject disturbances and converge to the NE found in \cite{sp16}. Next, assume each agent has partial information over a random  graph, $G_c$, Fig. \ref{fig:GraphOSNR}. The results of dynamics  (\ref{eq:AgentPartialInfo}) are plotted in  Fig. \ref{fig:5c}, while those of the Laplacian-based gradient dynamics \eqref{eq:GradDynPartialInfo} are shown in  Fig. \ref{fig:5b}, with similar comparison.

\begin{figure}[h!]
\vspace{-0.26cm}
	\centering
	\includegraphics[trim=0cm 0cm 0cm 0cm,width=2.2in]{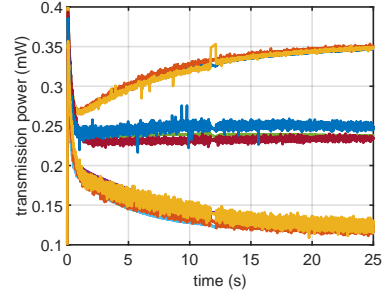}
	\caption{Laplacian-based dynamics \eqref{eq:GradDynPartialInfo} over $\!G_c\!$}\label{fig:5b} 
	\vspace{-0.26cm}
\end{figure}
\begin{figure}[h!]
\vspace{-0.26cm}
	\centering
	\includegraphics[trim=0cm 0cm 0cm 0cm,width=2.2in]{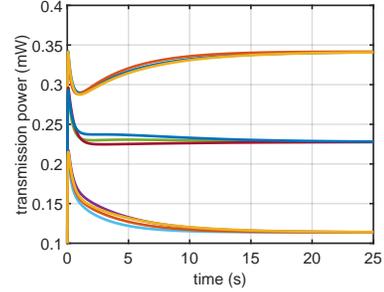}
	\caption{Agent dynamics  $\!\Sigma_i\!$ \eqref{eq:AgentPartialInfo} subject to disturbances over  $\!G_c\!$}\label{fig:5c}
	\vspace{-0.26cm}
\end{figure}

\vspace{-0.26cm}
\subsection{Sensor Networks}
Our next example is similar to the one investigated in \cite{sjs12}. However, our algorithm uses a continuous-time gradient-play inspired feedback instead of the discrete-time extremum seeking algorithm used in \cite{sjs12}. It is also important to note that while \cite{sjs12} considers noisy feedbacks, it does not consider disturbance rejection as we have posed it here.

Consider a group of five mobile robots in the plane in a sensor network. Each agent has a cost function that is a function of all robots' positions, $(x_i,x_{-i})$, 
\vspace{-0.25cm}
\begin{align}
J_i(x_i,x_{-i}) = x_i^Tx_i+x_i^Tr_i+\sum_{j \in \mathcal{I}} \|x_i-x_j\|^2
\end{align}
where $r_1 \!=\! \col(2,-2)$, $r_2 \!=\! \col(-2,-2)$, $r_3 \!=\! \col(-4,2)$, $r_4\! =\! \col(2,-4)$, and $r_5 \! =\! \col(3,3)$. We consider two types, 
velocity actuated and force-actuated robots, and  
in each case we consider the full-information and the partial-information case with communication over a random graph, $G_c$ (Fig. \ref{fig:Gc11}).
\begin{figure}[h!]
	\vspace{-0.26cm}
	\centering
	\includegraphics[trim=0cm 0cm 0cm 0cm,width=1.5in]{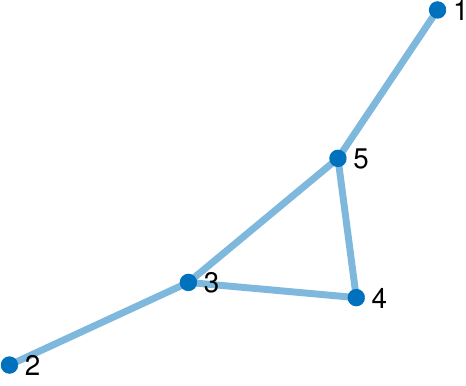}
	\caption{Random Communication Graph, $G_c$}\label{fig:Gc1}
\end{figure}
\subsubsection{Velocity-Actuated Robots}
Consider that each agent in the network is a velocity-actuated robot with dynamics given by (\ref{eq:Case1a}), 
where $d_i = \col(0.5,0)$ is a constant disturbance. 


\begin{figure}[h!]
 \vspace{-0.3cm}	\centering
	\includegraphics[trim=0cm 0cm 0cm 0cm,width=2.8in]{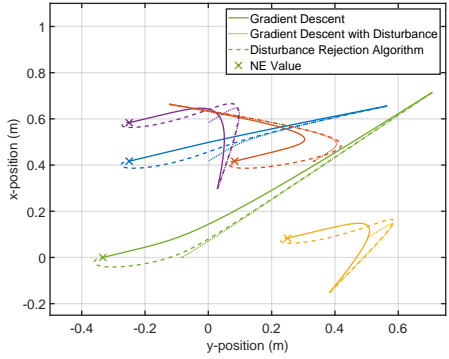}
	\caption{Comparison of \eqref{eq:AgentFullInfo} and \eqref{eq:GradDyn} for single-integrator agents}\label{fig:1}
\end{figure}

We consider first that each agent has full information about the other's actions and compare our algorithm \eqref{eq:AgentFullInfo} to the standard gradient-play \eqref{eq:GradDyn}. In Fig. \ref{fig:1} solid-lines depict gradient-play results in the disturbance-free case. In the presence of disturbances, as seen in Fig. \ref{fig:1}, \eqref{eq:AgentFullInfo} (dashed-lines) converges to the same NE values, while the standard gradient-play (dotted-lines) does not. 
\begin{figure}[h!]
 \vspace{-0.3cm}		\centering
	\includegraphics[trim=0cm 0cm 0cm 0cm,width=2.8in]{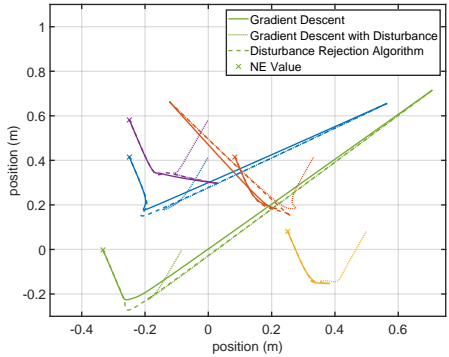}
	\caption{Comparison of $\!$ \eqref{eq:AgentPartialInfo} $\&$ \eqref{eq:GradDynPartialInfo} 
	for single-integrator agents}\label{fig:2}
	 \vspace{-0.3cm}	
\end{figure}
Next consider that each agent only has partial-information communicated over a graph, $G_c$. We compare our algorithm \eqref{eq:AgentPartialInfo} to that of the Laplacian based-gradient dynamics \eqref{eq:GradDynPartialInfo} in Fig. \ref{fig:2}, where solid-lines depict  \eqref{eq:GradDynPartialInfo} results in the disturbance-free case. In the presence of disturbances, Fig. \ref{fig:2} shows that   \eqref{eq:AgentPartialInfo} (dashed-lines) converges to the same NE values as found by the full-information case, Fig. \ref{fig:1}, while  \eqref{eq:GradDynPartialInfo} (dotted-lines) does not. 
\subsubsection{Force-Actuated Robots} Consider that each agent is modelled as double integrator, \eqref{eq:DoubleIntegratorDist} 
where $d_i = \col(0.5,0)$. 
The corresponding results  are shown in Fig. \ref{fig:3} (full information case) and  Fig. \ref{fig:4} (partial-information over $G_c$), where dashed-lines correspond to \eqref{eq:doubleIntDynDist} and    \eqref{eq:doubleIntDynDistPart}, respectively, while dotted-lines to the disturbance-free learning algorithm.

\begin{figure}[h!]
 \vspace{-0.3cm}	
	\centering
	\includegraphics[trim=0cm 0cm 0cm 0cm,width=2.8in]{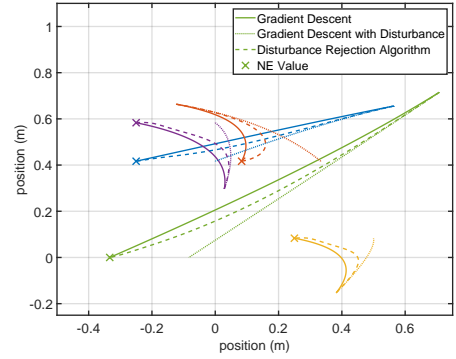}
	\caption{Results 
	of  \eqref{eq:doubleIntDynDist} for double-integrator agents}\label{fig:3}
	 \vspace{-0.3cm}	
\end{figure}

\begin{figure}[h!]
 \vspace{-0.3cm}	
	\centering
	\includegraphics[trim=0cm 0cm 0cm 0cm,width=2.8in]{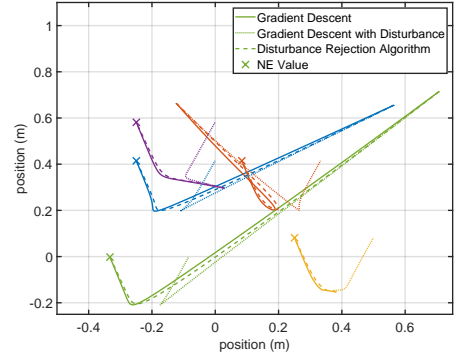}
	\caption{Results of \eqref{eq:doubleIntDynDistPart} for double-integrator agents}\label{fig:4}  
	 \vspace{-0.3cm}	
\end{figure}

\begin{remark}
Although we did not specifically investigate systems with noisy feedbacks, it is possible to show that due to their ISS properties, the dynamics \eqref{eq:AgentPartialInfo} and \eqref{eq:doubleIntDynDistPart} have a certain amount of robustness to feedback noise, such as the type investigated in \cite{sjs12} and \cite{ms17}. The ISS property implies that for any bounded feedback noise, the steady-state solution will remain in a neighbourhood of the NE.
\end{remark}

\vspace{-0.3cm}
\section{Conclusions}\label{sec:conclusions}
We considered Nash equilibrium seeking schemes for (multi)-integrator agents subject to external disturbances. We addressed the case of  full information  on the others' decisions, as well as the case where agents have partial-decision information, based on local observation and communication.  In both cases, we proposed new   continuous-time dynamic schemes that converge to the Nash equilibrium, irrespective of the disturbance.  Besides a gradient-play component,  the proposed agent dynamics have a dynamic internal-model component, 
and, in the case of partial-information,  a consensus component that drives agents to reach the decision-estimate consensus subspace. 


\bibliographystyle{IEEETran}
\bibliography{references}

\end{document}